\documentclass[12pt]{article}
\usepackage{lipsum}

\usepackage{epsfig}
\usepackage{amssymb}
\usepackage{color}
\usepackage{enumerate}
\usepackage{xmpmulti}
\usepackage{amssymb}
\usepackage{rotating}
\usepackage{amsmath}
\usepackage{amsthm}
\usepackage{amscd}
\usepackage[square,comma,sort&compress,numbers]{natbib} 
\usepackage{epsfig}
\usepackage{color}
\usepackage[latin5]{inputenc}
\usepackage{psfrag}
\usepackage{graphicx}

\newcommand\cL{{\mathcal L}}

\allowdisplaybreaks[1]
\theoremstyle{plain}

\newtheorem{theorem}{Theorem}[section]
\newtheorem{properties}[theorem]{Properties}
\newtheorem{lemma}[theorem]{Lemma}

\newtheorem*{main thm}{Main Theorem}

\theoremstyle{definition}
\newtheorem{definition}[theorem]{Definition}
\newtheorem{example}[theorem]{Example}

\newtheorem{remark}[theorem]{Remark}

\makeatletter
\def\@addpunct#1{%
  \relax\ifhmode
    \ifnum\spacefactor>\@m \else#1\fi
  \fi}
\newcommand{\keywordsname}{$^*$Correspondence:}
\def\@setkeywords{%
  {\keywordsname}\enspace \@keywords\@addpunct.}
\def\keywords#1{\def\@keywords{#1}}
\let\@keywords=\@empty
\g@addto@macro{\maketitle}{\begingroup%
  \let\@makefnmark\relax  \let\@thefnmark\relax%
  \ifx\@keywords\@mpty\else\@footnotetext{\@setkeywords}\fi%
  \endgroup}
\makeatletter
\date{}
\title{\large{Integral laminations on non-orientable surfaces}}
\author{\footnotesize{S.\"{O}yk\"{u} YURTTA\c{S}$^{\scriptstyle{1,*}}$, Mehmetcik PAMUK$^{\scriptstyle{1}}$}\\
$\scriptstyle{1}$ \footnotesize{Department of Mathematics, Dicle University, 21280
D\.{ı}yarbak{\i}r, Turkey}\\ \small{saadet.yurttas@dicle.edu.tr}\\
$\scriptstyle{2}$ \footnotesize{Department of Mathematics, Middle East Technical University,
Ankara, Turkey}\\ \footnotesize{mpamuk@metu.edu.tr}}

\keywords{saadet.yurttas@dicle.edu.tr}
\begin{document}
\maketitle

\begin{abstract}

We describe triangle coordinates for integral laminations on a non-orientable surface $N_{k,n}$ of genus $k$ with $n$ punctures and one boundary component, and give an explicit bijection from the set of integral laminations on $N_{k,n}$ to $(\mathbb{Z}^{2(n+k-2)}\times \mathbb{Z}^k)\setminus \left\{0\right\}$. 
 \end{abstract}

\noindent \small{\textbf{Keywords:} non-orientable surfaces, triangle coordinates, Dynnikov coordinates\\  AMS \it{Mathematics Subject classification}: 57N05, 57N16, 57M50}

\section{Introduction}
Let $N_{k, n}$ be a non-orientable surface  of genus $k$ with $n$ punctures and one boundary component. In this paper  we shall describe the generalized Dynnikov coordinate system for  the set of integral laminations $\mathcal{L}_{k,n}$, and give an explicit bijection between $\mathcal{L}_{k,n}$ and $(\mathbb{Z}^{2(n+k-2)}\times \mathbb{Z}^k)\setminus \left\{0\right\}$. To be more specific, we shall first take a particular collection of 
$3n+2k-4$ arcs and $k$ curves embedded in $N_{k,n}$, and describe each integral lamination by an element of $\mathbb{Z}_{\geq 0}^{3n+2k-4}\times \mathbb{Z}^k$,  its geometric intersection numbers with these arcs and curves. \emph{Generalized Dynnikov coordinates} are certain linear combinations of these integers that provide a one-to-one correspondence between $\mathcal{L}_{k,n}$ and $(\mathbb{Z}^{2(n+k-2)}\times \mathbb{Z}^k)\setminus \left\{0\right\}$.

The motivation for this paper comes from a recent work of Papadopoulos and Penner \cite{pepa} where they provide analogues for non-orientable surfaces of several results from Thurston theory of surfaces  which were studied only for orientable surfaces before \cite{W88, FLP79}.  Here we shall give the analogy of the Dynnikov Coordinate System \cite{D02} on the finitely punctured disk which has several useful applications such as giving an efficient method for  the solution of the word problem of the $n$-braid group \cite{or08},  computing the geometric intersection number of integral laminations \cite{paper2}, and counting the number of components they contain \cite{HY16}.

	Throughout the text we shall work on a standard model of  $N_{k, n}$ as illustrated in Figure \ref{arcsproof} where a disc with a cross drawn within it represents a crosscap, that is the interior of the disc is removed and the antipodal points on the resulting boundary component are identified (i.e. the boundary component bounds a M\"{o}bius band). 
	
The structure of the paper is as follows. In Section \ref{terminology} we give the necessary terminology and background. In Section~$2$ we describe and study the triangle coordinates for integral laminations on $N_{k, n}$, and  construct the generalized Dynnikov Coordinate System  giving the bijection $\rho \colon \cL_{k,n}\to(\mathbb{Z}^{2(n+k-2)}\times \mathbb{Z}^k)\setminus \left\{0\right\}$. An explicit formula for the inverse of this bijection is given in Theorem~\ref{thm:dynninvert}.

\subsection{Basic terminology and background}\label{terminology}
A simple closed curve in $N_{k,n}$  is inessential if it bounds an unpunctured disk, once punctured disk, or an unpunctured annulus. It is called essential otherwise. A simple closed curve is called \emph{2-sided} (respectively \emph{1-sided}) if a regular neighborhood of the curve is an annulus (respectively M\"{o}bius band). We say that a 2-sided curve is \emph{non-primitive} if it bounds a M\"{o}bius band 
\cite{pepa}, and a 1-sided curve is \emph{non-primitive} if it is a core curve of a M\"{o}bius band. They are called \emph{primitive} otherwise.

An integral lamination $\mathcal{L}$ on $N_{k, n}$  is a disjoint union of finitely many essential simple closed curves in $N_{k, n}$  modulo isotopy. Let $\mathcal{A}_{k,n}$ be the set of arcs $\alpha_i~(1\leq i\leq 2n-2)$,  $\beta_i~(1\leq i\leq n+k-1)$, $\gamma_i~(1\leq i\leq k-1)$ which have 
each endpoint either on the boundary or at a puncture, and the curves $c_i$~($1\leq i \leq k$) which are the core curves of M\"{o}bius bands in $N_{k,n}$ as illustrated in Figure \ref{arcsproof}:  the arcs $\alpha_{2i-3}$ and  $\alpha_{2i-2}$ for $2\leq i\leq n$  join the $i$-th puncture 
to $\partial N_{k,n}$, the arc $\beta_i$ has both end points on $\partial N_{k,n}$ and passes between the $i$-th and $(i+1)$-st  punctures for $1\leq i\leq n-1$,  the $n$-th puncture and the first crosscap for $i=n$, and the $(i-n)$-th and $(i+1-n)$-th  crosscaps for  $n+1\leq i\leq n+k-1$.  The arc $\gamma_i$ ( $1\leq i\leq k-1$) has both endpoints on $\partial N_{k,n}$ and surrounds the $i$-th crosscap.  
\begin{figure}[h!]
\begin{center}
\psfrag{c1}[tl]{$\scriptstyle{c_1}$} 
\psfrag{c2}[tl]{$\scriptstyle{c_2}$} 
\psfrag{ci}[tl]{$\scriptstyle{c_{i}}$} 
\psfrag{ck}[tl]{$\scriptstyle{c_k}$} 
\psfrag{gi}[tl]{$\scriptstyle{\gamma_i}$} 
\psfrag{a1}[tl]{$\scriptstyle{\alpha_1}$} 
\psfrag{g1}[tl]{$\scriptstyle{\gamma_1}$} 
\psfrag{g2}[tl]{$\scriptstyle{\gamma_2}$} 
\psfrag{g3}[tl]{$\scriptstyle{\gamma_{k-1}}$} 
\psfrag{d}[tl]{$\scriptstyle{\partial(N_{1,n})}$} 
\psfrag{w}[tl]{$\scriptstyle{w}$} 
\psfrag{i+1}[tl]{$\scriptstyle{i}$} 
\psfrag{i}[tl]{$\scriptstyle{i-1}$} 
\psfrag{n}[tl]{$\scriptstyle{n}$} 
\psfrag{1}[tl]{$\scriptstyle{\alpha_1}$} 
\psfrag{b3}[tl]{\begin{turn}{-90}$\scriptstyle{\textcolor{red}{\beta_{n}}}$\end{turn}} 
\psfrag{b7}[tl]{\begin{turn}{-90}$\scriptstyle{\textcolor{red}{\beta_{n+k-1}}}$\end{turn}} 
\psfrag{b1}[tl]{\begin{turn}{-90}$\scriptstyle{\textcolor{red}{\beta_{1}}}$\end{turn}} 
\psfrag{b4}[tl]{\begin{turn}{-90}$\scriptstyle{\textcolor{red}{\beta_{n+1}}}$\end{turn}} 
\psfrag{bi}[tl]{\begin{turn}{-90}$\scriptstyle{\textcolor{red}{\beta_{i}}}$\end{turn}} 
\psfrag{bi-1}[tl]{\begin{turn}{-90}$\scriptstyle{\textcolor{red}{\beta_{i-1}}}$\end{turn}} 
\psfrag{b6}[tl]{\begin{turn}{-90}$\scriptstyle{\textcolor{red}{\beta_{n+i}}}$\end{turn}} 
\psfrag{b5}[tl]{\begin{turn}{-90}$\scriptstyle{\textcolor{red}{\beta_{n+i-1}}}$\end{turn}} 
\psfrag{2i}[tl]{$\scriptstyle{\alpha_{2i}}$}
\psfrag{2}[tl]{$\scriptstyle{\alpha_2}$} 
\psfrag{y}[tl]{$\scriptstyle{y}$}
\psfrag{2i-5}[tl]{$\scriptstyle{\alpha_{2i-3}}$}
\psfrag{2i-2}[tl]{$\scriptstyle{\alpha_{2i}}$}
\psfrag{2i-3}[tl]{$\scriptstyle{\alpha_{2i-1}}$}
\psfrag{2i-1}[tl]{$\scriptstyle{\alpha_{2i-1}}$}
\psfrag{2i-4}[tl]{$\scriptstyle{\alpha_{2i-2}}$}
\psfrag{2n-4}[tl]{$\scriptstyle{\alpha_{2n-2}}$} 
\psfrag{2n-5}[tl]{$\scriptstyle{\alpha_{2n-3}}$} 
\psfrag{dip}[tl]{\begin{turn}{-90}$\scriptscriptstyle{\textcolor{blue}{\Delta'_i}}$\end{turn}} 
\psfrag{d1}[tl]{\begin{turn}{-90}$\scriptscriptstyle{\textcolor{blue}{\Delta_1}}$\end{turn}} 
\psfrag{d1p}[tl]{\begin{turn}{-90}$\scriptscriptstyle{\textcolor{blue}{\Delta'_1}}$\end{turn}} 
\psfrag{d2p}[tl]{\begin{turn}{-90}$\scriptscriptstyle{\textcolor{blue}{\Delta'_2}}$\end{turn}} 
\psfrag{d3p}[tl]{\begin{turn}{-90}$\scriptscriptstyle{\textcolor{blue}{\Delta'_{k-1}}}$\end{turn}} 
\psfrag{d4p}[tl]{\begin{turn}{-90}$\scriptscriptstyle{\textcolor{blue}{\Delta'_{k}}}$\end{turn}} 
\psfrag{d2i-2}[tl]{\begin{turn}{-90}$\scriptscriptstyle{\textcolor{blue}{\Delta_{2i}}}$\end{turn}} 
\psfrag{d2i-1}[tl]{\begin{turn}{-90}$\scriptscriptstyle{\textcolor{blue}{\Delta_{2i-1}}}$\end{turn}} 
\psfrag{d2n-4}[tl]{\begin{turn}{-90}$\scriptscriptstyle{\textcolor{blue}{\Delta_{2n-4}}}$\end{turn}} 
\psfrag{d2n-3}[tl]{\begin{turn}{-90}$\scriptscriptstyle{\textcolor{blue}{\Delta_{2n-2}}}$\end{turn}} 
\psfrag{d2i-4}[tl]{\begin{turn}{-90}$\scriptscriptstyle{\textcolor{blue}{\Delta_{2i-2}}}$\end{turn}} 
\psfrag{d2i-3}[tl]{\begin{turn}{-90}$\scriptscriptstyle{\textcolor{blue}{\Delta_{2i-1}}}$\end{turn}} 
\psfrag{d2i-5}[tl]{\begin{turn}{-90}$\scriptscriptstyle{\textcolor{blue}{\Delta_{2i-3}}}$\end{turn}} 
\psfrag{d0}[tl]{\begin{turn}{-90}$\scriptscriptstyle{\textcolor{blue}{\Delta_{0}}}$\end{turn}} 
\psfrag{o1}[tl]{\begin{turn}{-90}$\scriptscriptstyle{\textcolor{blue}{\Sigma_1}}$\end{turn}} 
\psfrag{oi}[tl]{\begin{turn}{-90}$\scriptscriptstyle{\textcolor{blue}{\Sigma_i}}$\end{turn}} 
\includegraphics[width=0.97\textwidth]{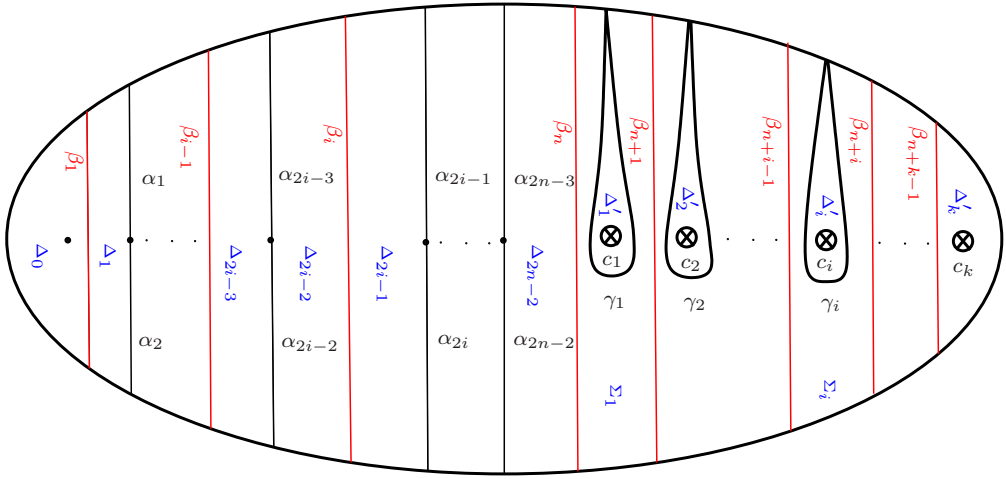}
%\resizebox{0.5\textwidth, angle=-90}{!}{\includegraphics{famcurves}}
\caption{The arcs $\alpha_i$, $\beta_i$, $\gamma_i$, the 1-sided curves $c_1,c_2,\dots, c_k$ and the regions $\Delta_i$ and $\Sigma_i$.} 
\label{arcsproof}
\end{center}
\end{figure}

The surface is divided by these arcs into $2n+2k-2$ regions, $2n+k-3$ of these are triangular since each $\Delta_i$ ($1\leq i \leq 2n-2 $) and $\Sigma_i$ ($1\leq i \leq k-1$) is bounded by three arcs when the boundary of the surface is identified to a point.  The two triangles $\Delta_{2i-3}$ and 
$\Delta_{2i-2}$ on the left and right hand side of the $i$-th puncture are defined by the arcs $\alpha_{2i-3}, \alpha_{2i-2}, \beta_{i-1}$ and  $\alpha_{2i-3}, \alpha_{2i-2}, 
\beta_{i}$ respectively. The triangle $\Sigma_{i}$ is 
defined by the arcs $\gamma_{i}, \beta_{n+i-1}, \beta_{n+i}$. Each $\Delta'_i$ ($1\leq i\leq k-1$) is bounded by $\gamma_i$, and the two end regions $\Delta_0$ and $\Delta'_k$ are bounded by $\beta_1$ and $\beta_{n+k-1}$ respectively.  
Given $\mathcal{L}\in \mathcal{L}_{k,n}$, let $L$ be a taut representative of $\mathcal{L}$ with respect to the elements of $\mathcal{A}_{k,n}$. That is, $L$ intersects each of the arcs  and curves in  $\mathcal{A}_{k,n}$ minimally.

\begin{figure}[h!]
\centering
\psfrag{y}[tl]{$\scriptstyle{y}$} 
\psfrag{w}[tl]{$\scriptstyle{w}$} 
\psfrag{3}[tl]{$\scriptstyle{\beta_{n-1}}$} 
\psfrag{d2n-3}[tl]{\textcolor{blue}{$\scriptstyle{\Delta_{2n-3}}$}}
\psfrag{1}[tl]{$\scriptstyle{\alpha_{2n-3}}$}
\psfrag{2}[tl]{$\scriptstyle{\alpha_{2n-2}}$}
\psfrag{2i-1}[tl]{$\scriptstyle{\alpha_{2i-1}}$} 
\psfrag{i+1}[tl]{$\scriptstyle{\beta_{i+1}}$} 
\psfrag{i}[tl]{$\scriptstyle{\beta_{i}}$} 
\psfrag{gi}[tl]{$\scriptstyle{\gamma_{i}}$} 
\psfrag{b1}[tl]{$\scriptstyle{\beta_{n+i-1}}$} 
\psfrag{b2}[tl]{$\scriptstyle{\beta_{n+i}}$} 
\psfrag{2i}[tl]{$\scriptstyle{\alpha_{2i}}$}
\psfrag{d1}[tl]{$\scriptstyle{\Delta_{2i-1}}$} 
\psfrag{d2}[tl]{$\scriptstyle{\Delta_{2i}}$} 
\psfrag{si}[tl]{$\scriptstyle{S_{i}}$} 
\psfrag{oi}[tl]{$\scriptstyle{\Sigma_{i}}$} 
\includegraphics[width=0.5\textwidth]{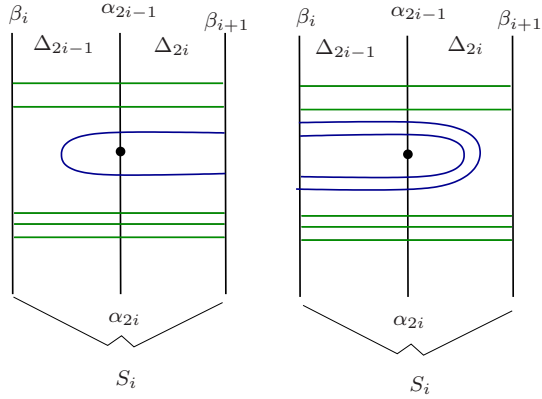}
%\resizebox{0.5\textwidth, angle=-90}{!}{\includegraphics{famcurves}}
\caption{There is 1 left loop component in the first case and 2 right loop components in the second case. There are 2 above and 3 below components in each case.} 
\label{pathcomponentsproof1}
\end{figure}

\begin{definition}\label{abovebelowdef1}
Set  $S_i=\Delta_{2i-1}\cup\Delta_{2i}$ for each $i$ with $1\leq i\leq n-1$. A \emph{path component} 
of $L$ in $S_i$ is a component of $L\cap S_i$.  There are four types of path components in $S_i$ as depicted in Figure~\ref{pathcomponentsproof1}:
  
\begin{itemize}
\setlength\itemsep{0.1em}
\item An \emph{above component} has end points on $\beta_i$ and $\beta_{i+1}$, passing across $\alpha_{2i-1}$,
\item A \emph{below component} has end points on $\beta_i$ and $\beta_{i+1}$,  passing across $\alpha_{2i}$,
\item A \emph{left loop component} has both end points on $\beta_{i+1}$,
\item A \emph{right loop component} has both end points on $\beta_{i}$.
\end{itemize}

\end{definition}

 \begin{figure}[ht!]
\centering
\psfrag{y}[tl]{$\scriptstyle{y}$} 
\psfrag{w}[tl]{$\scriptstyle{w}$} 
\psfrag{3}[tl]{$\scriptstyle{\beta_{n-1}}$} 
\psfrag{d2n-3}[tl]{\textcolor{blue}{$\scriptstyle{\Delta_{2n-3}}$}}
\psfrag{1}[tl]{$\scriptstyle{\alpha_{2n-3}}$}
\psfrag{2}[tl]{$\scriptstyle{\alpha_{2n-2}}$}
\psfrag{2i-1}[tl]{$\scriptstyle{\alpha_{2i-1}}$} 
\psfrag{i+1}[tl]{$\scriptstyle{\beta_{i+1}}$} 
\psfrag{i}[tl]{$\scriptstyle{\beta_{i}}$} 
\psfrag{gi}[tl]{$\scriptstyle{\gamma_{i}}$} 
\psfrag{b1}[tl]{$\scriptstyle{\beta_{n+i-1}}$} 
\psfrag{b2}[tl]{$\scriptstyle{\beta_{n+i}}$} 
\psfrag{2i}[tl]{$\scriptstyle{\alpha_{2i}}$}
\psfrag{d1}[tl]{$\scriptstyle{\Delta_{2i-1}}$} 
\psfrag{d2}[tl]{$\scriptstyle{\Delta_{2i}}$} 
\psfrag{si}[tl]{$\scriptstyle{S_{i}}$} 
\psfrag{oi}[tl]{$\scriptstyle{S^{'}_i}$} 
\psfrag{ci}[tl]{$\scriptstyle{c_{i}}$} 
\includegraphics[width=0.85\textwidth]{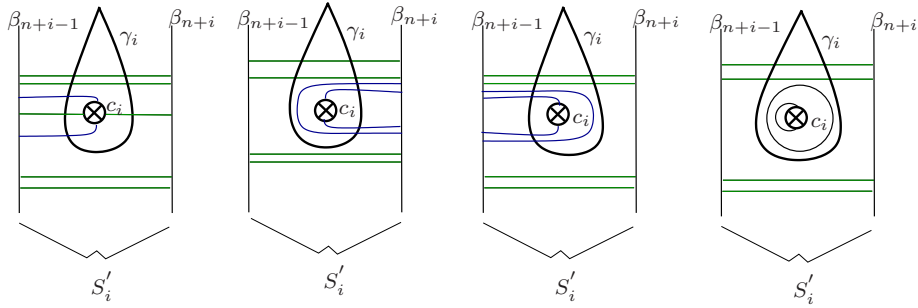}
%\resizebox{0.5\textwidth, angle=-90}{!}{\includegraphics{famcurves}}
\vspace{1mm}
\caption{There is $1$ right core loop and $1$ straight core component in the first case;  $1$ left loop and $1$ left core loop component in the second case;  $1$ right non-core loop and $1$ right core loop component in the third case and $1$ $1$-sided and $1$ $2$-sided non-primitive curves in the fourth case. There are 2 above and 2 below components in each case.} 
\label{pathcomponentsproof2} 
\end{figure}

\begin{definition}\label{abovebelowdef2}
Set $S'_i=\Delta'_{i}\cup\Sigma_i$ for each  $1\leq i\leq k-1$. A \emph{path component} 
of $L$ in $S'_i$ is a component of $L\cap S'_i$.  There are $7$ types of path components in $S^{'}_i$ as depicted in Figure  \ref{pathcomponentsproof2}.
\begin{itemize}
\setlength\itemsep{0.1em}
\item An \emph{above component} has end points on $\beta_{n+i-1}$ and $\beta_{n+i}$, and passes across $\gamma_{i}$ without intersecting $c_i$,
\item A \emph{below component} has end points on $\beta_{n+i-1}$ and $\beta_{n+i}$, and doesn't pass across $\gamma_{i}$,
\item A \emph{left loop component} has both end points on $\beta_{n+i}$,
\item A \emph{right loop component} has both end points on $\beta_{n+i-1}$,

If a loop component intersects $c_i$, it is called \emph{core loop component} otherwise it is called non-core loop component.
\item A \emph{straight core component} has end points on $\beta_{n+i-1}$ and $\beta_{n+i}$, and intersects $c_i$,

\item A non-primitive $1$-sided curve,

	If $L$ contains a non-primitive $1$-sided curve $c_i$ we depict it with a ring with end points on the $i$-th crosscap as shown in the fourth case in Figure \ref{pathcomponentsproof2}.

\item A non-primitive $2$-sided curve.
\end{itemize}

\end{definition}

%%%%%%%%%%%%%%%%%%%%%%%%%%%%%%%%%%%%%%%%%%%%%%%%%%%%%%%%%%%%%%%%%%%%%%%%%%%%%%%%%%%%%%%%%%%%%%%%%%%%%%%%%%%%%%%%%%%%%%%%%%%%%%%%%%

\section{Triangle coordinates}
Let $L$ be a taut representative of $\mathcal{L}$. Write $\alpha_i, \beta_i, \gamma_i$ and $c_i$  for the geometric intersection number of $L$ with the arc  $\alpha_i, \beta_i, \gamma_i$ and the core curve $c_i$ respectively.  It will always be clear from the context whether we mean the arc or the geometric intersection number assigned on the arc.

\begin{definition}
The triangle coordinate function $\tau \colon \mathcal{L}_{k,n}\to (\mathbb{Z}_{\geq 0}^{3n+2k-4}\times \mathbb{Z}^k)\setminus\left\{0\right\}$  
is defined by  
\begin{align*}
\tau(\mathcal{L})&=(\alpha_1, \ldots, \alpha_{2n-2}; \beta_1, \ldots, \beta_{n+k-1}; \gamma_1, \ldots, \gamma_{k-1}; c_1, \ldots, c_k).
\end{align*}
\noindent where  $c_i=-1$ if $L$ contains the $i$-th core curve;  $c_i=-2m$ if it contains $m\in \mathbb{Z}^+$ disjoint copies  2-sided non-primitive curves around the $i$-th crosscap, and  $c_i=-2m-1$ if it contains $m$ disjoint copies of 2-sided non-primitive curves around the $i$-th crosscap plus the $i$-th core curve. 

\end{definition}

\begin{remark}\label{sumofloop}
Let $b_i=\frac{\beta_{i}-\beta_{i+1}}{2}$ for $1\leq i \leq n+k-2$. Then in each $S_i$ ($1\leq i \leq n-1$) and $S'_i$ ($n\leq i \leq n+k-2$) there are $|b_i|$ loop components. Furthermore, if $b_i<0$ these loop components are left, and if $b_i>0$ they are right.  \end{remark}

The proof of the next lemma is obvious from Figure \ref{pathcomponentsproof1}.

\begin{lemma}\label{abovebelow}
Let $1\leq i \leq n-1$. The number of above and below components in $S_i$ are given by $a_{S_i}=\alpha_{2i-1}-|b_i|$ and $b_{S_i}=\alpha_{2i}-|b_i|$ respectively. 
\end{lemma}

Let $\lambda_{i}$ and $\lambda_{c_i}$ denote the number of non-core and core loop components, $\psi_{i}$ the number of straight core components, and $a_{S'_i}$ and $b_{S'_i}$ the number of above and below components in $S'_i$.

\begin{lemma}\label{abovebelow2}
Let $L$ be a taut representative of $\mathcal{L}\in \mathcal{L}_{k,n}$, and set $c^+_i=\max(c_i,0)$. Then for each $1\leq i \leq k-1$ we have

\begin{align*}
\lambda_{i}&=\max(|b_{n+i-1}|-c^+_i,0),                        &&\lambda_{c_i}= \min(|b_{n+i-1}|,c^+_i),\\
\psi_{i}&=\max(c^+_i-|b_{n+i-1}|,0).
\end{align*}
\end{lemma}

\begin{proof}
Assume that $L$ doesn't contain any non-primitive curve in $S'_i$. Since $c_i$ gives the sum of straight core and core loop components and $|b_{n+i-1}|$ gives the sum of non-core loop and core loop components in $S'_i$ 
(see Figure \ref{pathcomponentsproof2}) we have 
\begin{equation} \label{eq:coreloop}
c_i=\psi_{i}+\lambda_{c_i}  \hspace{0.5cm} \text{and} \hspace{0.5cm}  |b_{n+i-1}|=\lambda_{i}+\lambda_{c_i}.
\end{equation}

	If $c_i>|b_{n+i-1}|$, then clearly there exists a straight core component in $S'_i$, and hence no non-core loop component in $S'_i$ that is $\lambda_{i}=0$. Therefore in this case, $\lambda_{c_i}=|b_{n+i-1}|$ and hence $\psi_{i}=c_i-|b_{n+i-1}|$ by Equation \ref{eq:coreloop}. 
	
		If $c_i<|b_{n+i-1}|$, there exists a non-core loop component in $S'_i$, and hence no straight core components in $S'_i$ that is $\psi_{i}=0$.  Therefore $c_i=\lambda_{c_i}$ and hence $\lambda_i=|b_{n+i-1}|-c_i$ by Equation \ref{eq:coreloop}. We get
		
\begin{align*}
\lambda_{i}&=\max(|b_{n+i-1}|-c_i,0)\\                        
\psi_{i}&=\max(c_i-|b_{n+i-1}|,0).
\end{align*}

Also if $|b_{n+i-1}|<c_i, \lambda_i=0$ and hence $\lambda_{c_i}= |b_{n+i-1}|$, if $|b_{n+i-1}|>c_i, \psi_{i}=0$ and hence $\lambda_{c_i}=c_i$ by Equation \ref{eq:coreloop}. Therefore we get, $\lambda_{c_i}= \min(|b_{n+i-1}|,c_i)$.

		Finally, if $L$  contains a non-primitive curve in $S'_i$, there can be no straight core  and core loop component in $S'_i$ that is $\psi_i=\lambda_{c_i}=0$, hence $\lambda_i=|b_{n+i-1}|$. Since $c_i<0$ by definition, setting $c^+_i=\max(c_i,0)$ we can write  

\begin{align*}
\lambda_{i}&=\max(|b_{n+i-1}|-c^+_i,0),                        &&\lambda_{c_i}= \min(|b_{n+i-1}|,c^+_i),\\
\psi_{i}&=\max(c^+_i-|b_{n+i-1}|,0).
\end{align*}

\end{proof}

\begin{lemma}\label{abovebelow3}
Let $L$ be a taut representative of $\mathcal{L}\in \mathcal{L}_{k,n}$. For each $1\leq i \leq k-1$ we have
\begin{align*}
a_{S'_i}&=\frac{\gamma_{i}}{2}-|b_{n+i-1}|-\psi_i\\
b_{S'_i}&=\max(\beta_{n+i-1}, \beta_{n+i})-|b_{n+i-1}|-\frac{\gamma_i}{2}.
\end{align*}

\end{lemma}
\begin{proof}

To compute the number of above and below components in $S'_i$ we observe that each path component other than a below component in $S'_i$ intersects $\gamma_i$ twice, that is $ \gamma_i=2(a_{S'_i}+|b_{n+i-1}|+\psi_i)$. Therefore we get,
$$ a_{S'_i}=\frac{\gamma_{i}}{2}-|b_{n+i-1}|-\psi_i.$$
\noindent To compute the number of below components, we note that the sum of all path components in $S'_i$ is given by $\beta=\max(\beta_{n+i-1}, \beta_{n+i})$. Then $b_{S'_i}$ is $\beta$ minus  the  number of above, straight core components and twice the number loop components in ${S'_i}$ (each loop component intersects $\beta$ twice). We get
\begin{align*}b_{S'_i}&=\max(\beta_{n+i-1}, \beta_{n+i})-a_{S'_i}-2|b_{n+i-1}|-\psi_i\\
&=\max(\beta_{n+i-1}, \beta_{n+i})-|b_{n+i-1}|-\frac{\gamma_i}{2}\end{align*}\end{proof}

Another way of expressing $ a_{S'_i}$ and $b_{S'_i}$ is given in item P$4.$ in Properties \ref{properties}.

\begin{remark}\label{noncore}
Observe that the loop components in $\Delta_{0}$ are always left and the number of 
them is given by $\frac{\beta_{1}}{2}$. Similarly, the loop components in $\Delta'_{k}$ are always right and the number of core and non-core loop components in $\Delta'_{k}$  are given by $c_k$ and 
$\lambda_k=\frac{\beta_{n+k-1}}{2}-c_k$.
\end{remark}
Lemma \ref{equalities} and Lemma \ref{equalities2} are obvious from Figure \ref{pathcomponentsproof1} and Figure \ref{pathcomponentsproof2}.

\begin{lemma}\label{equalities}
There are equalities for each $S_i$: 
\begin{itemize}
\item When there are left loop components ($b_i<0$), 
\begin{align*}
\alpha_{2i}+\alpha_{2i-1}                    &=\beta_{i+1}\\
\alpha_{2i}+\alpha_{2i-1}-\beta_i        &=2|b_i|,
\end{align*}

\item When there are right loop components ($b_i>0$),
\begin{align*}
\alpha_{2i}+\alpha_{2i-1}                          &=\beta_{i}\\
\alpha_{2i}+\alpha_{2i-1}-\beta_{i+1}       &=2|b_i|,
\end{align*}

\item When there are no loop components ($b_i=0$),
\begin{align*}
\alpha_{2i}+\alpha_{2i-1}=\beta_i=\beta_{i+1}.
\end{align*}
\end{itemize}
\end{lemma}

\begin{lemma}\label{equalities2}
There are equalities for each $S'_i$: 
\begin{itemize}
\item When there are left loop components ($b_{n+i-1}<0$), 
\begin{align*} 
a_{S'_i}+ b_{S'_i}+\psi_{i}+2|b_{n+i-1}|&=\beta_{n+i}\\
a_{S'_i}+ b_{S'_i}+\psi_{i}&=\beta_{n+i-1}
\end{align*}
\item When there are right loop components ($b_{n+i-1}> 0$) 
\begin{align*}
a_{S'_i}+ b_{S'_i}+\psi_{i}+2|b_{n+i-1}|&=\beta_{n+i-1}.\\
a_{S'_i}+ b_{S'_i}+\psi_{i}&=\beta_{n+i}
\end{align*}
\item When there are no loop components $b_{n+i-1}=0$ 
\begin{align*}
a_{S'_i}+ b_{S'_i}+\psi_{i}=\beta_{n+i}=\beta_{n+i-1}.
\end{align*}
\end{itemize}
\end{lemma}

\begin{example}\label{firstexample}

Let $\tau(\cL)=(4, 2, 2, 6; 2, 6, 8, 4; 8; 1, 1)$ be the triangle coordinates of an integral lamination 
$\cL\in \mathcal{L}_{2,3}$. We shall show how we draw $\cL$ from 
its given triangle coordinates. First, we compute the loop components in the two end regions $\Delta_{0}$ and $\Delta'_{2}$ using Remark \ref{noncore}. Since $\beta_1=2$ there is one loop component in $\Delta_{0}$. Similarly, since $\beta_4=4$ and $c_2=1$, we get $\lambda_2=\frac{\beta_{4}}{2}-c_2=1$. 

	Next, we compute loop components in $S_1$, $S_2$ and $S'_1$. Since  $b_i=\frac{\beta_i-\beta_{i+1}}{2}$ for each $1\leq i\leq 3$ we have $b_1=-2, b_2=-1$. Hence there are two left loop components in $S_1$, and one left component in $S_2$. Similarly since $b_3=2$ there are $2$ right loop components in $S'_1$, and by Lemma \ref{abovebelow2},  $\lambda_{1}=\max(|b_{3}|-c_1,0)=1$ (hence $\psi_1=0$) and $\lambda_{c_1}= \min(|b_{3}|,c_1)=1$. Using Lemma~\ref{abovebelow} and Lemma \ref{abovebelow3} we compute the number of above and below  components. We get $a_{S_1}=\alpha_1-|b_1|=2$, $b_{S_1}=\alpha_2-|b_1|=0$, $a_{S_2}=\alpha_3-|b_2|=1$, $b_{S_2}=\alpha_4-|b_2|=5$, and 
\begin{align*}
a_{S'_1}&=\frac{\gamma_{1}}{2}-|b_{3}|-\psi_1=2\\
b_{S'_1}&=\max(\beta_{3}, \beta_{4})-|b_{3}|-\frac{\gamma_1}{2}=2.
\end{align*}

Connecting the path components in each $\Delta_{0}$, $\Delta'_{2}$, $S_1$, $S_2$ and $S'_1$ we draw the integral lamination as shown in Figure \ref{munu2proof}.
\end{example}

\begin{figure}[h!]
\begin{center}
\psfrag{0}[tl]{$\scriptstyle{0}$} 
\psfrag{b1}[tl]{$\scriptstyle{2}$} 
\psfrag{b2}[tl]{$\scriptstyle{6}$} 
\psfrag{b3}[tl]{$\scriptstyle{8}$}
\psfrag{b4}[tl]{$\scriptstyle{4}$} 
\psfrag{2i}[tl]{$\scriptstyle{5}$}
\psfrag{2i-5}[tl]{$\scriptstyle{4}$}
\psfrag{2i-2}[tl]{$\scriptstyle{6}$}
\psfrag{2i-3}[tl]{$\scriptstyle{2}$}
\psfrag{2i-1}[tl]{$\scriptstyle{1}$}
\psfrag{2i-4}[tl]{$\scriptstyle{2}$}
\psfrag{g1}[tl]{$\scriptstyle{8}$} 
\includegraphics[width=0.65\textwidth]{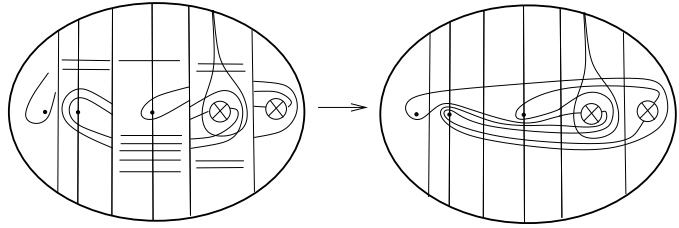}
%\resizebox{0.5\textwidth, angle=-90}{!}{\includegraphics{famcurves}}
\caption{$\tau(L)=(4, 2, 2, 6; 2, 6, 8, 4; 8; 1, 1)$}\label{munu2proof}
\end{center}
\end{figure}
\begin{lemma}\label{tauinject}
The triangle coordinate function  $\tau \colon \mathcal{L}_{k,n}\to (\mathbb{Z}_{\geq 0}^{3n+2k-4}\times \mathbb{Z}^k)\setminus\left\{0\right\}$ is injective.
\end{lemma}

\begin{proof}
We can determine the number of loop, above and below components in each $S_i$ by Remark \ref{sumofloop} and Lemma \ref{abovebelow}; core and non-core loop, straight core, above and below  components in each $S'_i$ by Lemma \ref{abovebelow2} and Lemma \ref{abovebelow3} as illustrated in Example \ref{firstexample}. The components in each $S_i$ and $S'_i$  are glued together in a unique way up to isotopy, and hence $\cL$ is constructed uniquely.\end{proof}

\begin{remark}
The triangle coordinate function $\tau \colon \mathcal{L}_{k,n}\to (\mathbb{Z}_{\geq 0}^{3n+2k-4}\times \mathbb{Z}^k)\setminus\left\{0\right\}$  is not surjective: an integral lamination must satisfy the triangle inequality in each $S_i$ and $S'_i$, and some additional conditions such as the equalities in Lemma~\ref{equalities} and Lemma~\ref{equalities2}. 
\end{remark}
Next, we give a list of properties an integral lamination $\mathcal{L}\in \mathcal{L}_{k,n}$ satisfies in terms of its triangle coordinates as in \cite{paper2}, and then construct a new coordinate system from the triangle coordinates which describes integral laminations in a unique way. In particular, we shall generalize the Dynnikov coordinate system \cite{D02, wiest, or08, HY08,paper2, paper3, HY16} for $N_{k,n}$. 

\begin{properties}\label{properties}
Let $L$ be a taut representative of $\mathcal{L}\in \mathcal{L}_{k,n}$. 

\begin{enumerate}
\item[\textnormal{P1}.]  Every component of $L$ intersects each $\beta_i$ an even number of times. Recall from Remark \ref{sumofloop} that the number of loop components is given by  $|b_i|$ where $b_i=\frac{\beta_i-\beta_{i+1}}{2}$.

\begin{figure}[h!]
\begin{center}
\psfrag{m1}[tl]{$\scriptstyle{m_i+x_i}$} 
\psfrag{m}[tl]{$\scriptstyle{m_i}$} 
\psfrag{a1}[tl]{$\scriptstyle{\alpha_{2i-1}}$} 
\psfrag{a2}[tl]{$\scriptstyle{\alpha_{2i}}$} 
\psfrag{b1}[tl]{$\scriptstyle{\beta_i}$} 
\psfrag{b2}[tl]{$\scriptstyle{\beta_{i+1}}$} 
\psfrag{g}[tl]{$\scriptstyle{\gamma_i}$} 
\psfrag{r}[tl]{$\scriptstyle{n_i}$} 
\psfrag{tr}[tl]{$\scriptstyle{t_i+n_i}$} 
\psfrag{bi}[tl]{$\scriptstyle{|b_{i}|}$} 
\psfrag{b}[tl]{$\scriptstyle{|b_{n+i-1}|}$} 
\psfrag{f}[tl]{$\scriptstyle{\psi_i}$}
\psfrag{x}[tl]{$\scriptstyle{\beta_{n+i-1}}$} 
\psfrag{y}[tl]{$\scriptstyle{\beta_{n+i}}$}
\psfrag{2i-5}[tl]{$\scriptstyle{4}$}
\psfrag{2i-2}[tl]{$\scriptstyle{6}$}
\psfrag{2i-3}[tl]{$\scriptstyle{2}$}
\psfrag{2i-1}[tl]{$\scriptstyle{1}$}
\psfrag{2i-4}[tl]{$\scriptstyle{2}$}
\psfrag{g1}[tl]{$\scriptstyle{8}$} 
\includegraphics[width=0.55\textwidth]{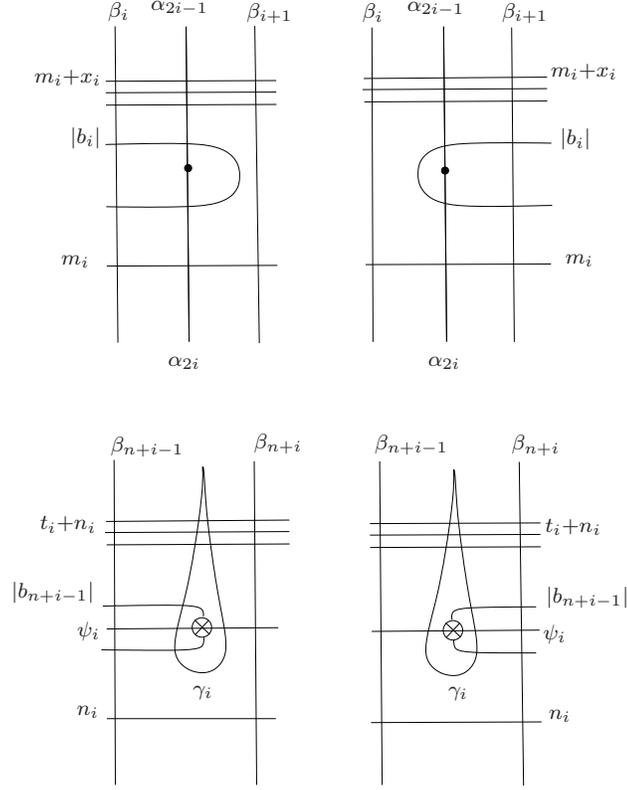}
%\resizebox{0.5\textwidth, angle=-90}{!}{\includegraphics{famcurves}}
\caption{$m_i$ and $n_i$ denote the smaller of above and below components in $S_i$ and $S'_i$ repectively}\label{relationII}
\end{center}
\end{figure}

\item[\textnormal{P2}.]  Set $x_i=|\alpha_{2i}-\alpha_{2i-1}|$ and $t_i=|a_{S'_i}-b_{S'_i}|$. Then $x_i$ and $t_i$ gives the difference between the number of above and below components in $S_i$ and $S'_i$ respectively. Set 
$m_i=\min\left\{\alpha_{2i}-|b_i|, \alpha_{2i-1}-|b_i|\right\}$; $1\leq i \leq  n-1$  and  $n_i= \min\left\{ a_{S'_i}, b_{S'_i} \right\}$; $1\leq i \leq k-1$.  See Figure \ref{relationII}. Note that $x_i$ is even since $L$ intersects $\alpha_{2i}\cup \alpha_{2i-1}$ an even number of times. Clearly, this may not hold for $t_i$ since when $\psi_i$ is odd the sum of above and below components (and hence their difference) is odd. See  Lemma~\ref{equalities2}.

\item[\textnormal{P3}.]   Set $2a_i=\alpha_{2i}-\alpha_{2i-1}$ $(|a_i|=x_i/2)$.  Then, by Lemma~\ref{equalities}  we get 

\begin{itemize}
\item If $b_i\geq 0$, then $\beta_i=\alpha_{2i} + \alpha_{2i-1}$ and hence
$$
\alpha_{2i}=a_i+\frac{\beta_i}{2}  \ \textrm{and} \  \alpha_{2i-1}=-a_i+\frac{\beta_i}{2}. 
$$

\item If $b_i \leq 0$, then $\beta_{i+1}=\alpha_{2i} + \alpha_{2i-1}$  and hence
 $$
\alpha_{2i}=a_i+\frac{\beta_{i+1}}{2}  \ \textrm{and} \  \alpha_{2i-1}=-a_i+\frac{\beta_{i+1}}{2}.
$$
\end{itemize}

That is,

\begin{align*}
\displaystyle
\alpha_i = \left\{     
\begin{array}{lr}
(-1)^i a_{\lceil i/2 \rceil}+\frac{\beta_{\lceil i/2\rceil}}{2}& \mbox{if $b_{\lceil i/2\rceil} \geq 0$}, \\
(-1)^i a_{\lceil i/2\rceil}+\frac{\beta_{1+\lceil i/2\rceil}}{2}& \mbox{if $b_{\lceil i/2\rceil} \leq 0$}.  
\end{array}
\right.\\
\end{align*}
where $\lceil i/2 \rceil$ denotes the smallest integer that is not less than $i/2$.

\item[\textnormal{P4}.]    Since $t_i=a_{S'_i}-b_{S'_i}$ for $1\leq i \leq k-1$,  from Lemma \ref{equalities2} we get

\begin{itemize}

\item  If $b_{n+i-1}\geq 0$ then $a_{S'_i}+b_{S'_i}+\psi_{i}+2b_{n+i-1}=\beta_{n+i-1}$, and 
$$
a_{S'_i}=\frac{t_i-\psi_i+\beta_{n+i-1}-2b_{n+i-1}}{2}
$$

\item  If $b_{n+i-1}\leq 0$ then $a_{S'_i}+b_{S'_i}+\psi_{i}-2b_{n+i-1}=\beta_{n+i}$, and 
$$
a_{S'_i}=\frac{t_i-\psi_{i}+\beta_{n+i}+2b_{n+i-1}}{2}
$$

And hence 
$$a_{S'_i} =\frac{t_i-\psi_{i}+\max(\beta_{n+i}, \beta_{n+i-1})-2|b_{n+i-1}|}{2}
$$ 
Similarly we compute $$b_{S'_i} =\frac{-t_i-\psi_{i}+\max(\beta_{n+i}, \beta_{n+i-1})-2|b_{n+i-1}|}{2}$$
\end{itemize}
\item[\textnormal{P5}.]   It is easy to observe from Figure \ref{relationII} that

\begin{align*}
\beta_i&=2\left[\left|a_i \right|+\max(b_i,0)+m_i\right] \quad \text{for}\quad 1\leq i \leq n-1\\
\beta_{n+i}&=|t_i|+2\max(b_{n+i-1}, 0)+\psi_{i}+2n_i \quad \text{for}\quad 1\leq i \leq k-1.
\end{align*}

\noindent Therefore, since $b_i=\frac{\beta_{i}-\beta_{i+1}}{2}$; $1\leq i \leq n+k-2$ we can compute $\beta_1$ using one of the two equations below:

\begin{align*}
\beta_1&=2\left[\left|a_i \right|+\max(b_i, 0)+m_i+\displaystyle\sum^{i-1}_{j=1}b_j \right]\quad  \text{for}\quad 1\leq i \leq n-1,\\
\beta_{1}&=|t_i|+2\max(b_{n+i-1}, 0)+\psi_{i}+2n_i +2\sum_{j=1}^{n+i-2}b_j\quad \text{for}\quad 1\leq i \leq k-1.
\end{align*} 

 \vspace{-0.5cm}

\begin{figure}[h!] %\label{fig:ucomponent}
\begin{center}
\psfrag{u}[tl]{\tiny{u-component}} 
\psfrag{r}[tl]{$\scriptstyle{m_i}$} 
\psfrag{tr}[tl]{$\scriptstyle{|2a_i|+m_i}$} 
\psfrag{b}[tl]{$\scriptstyle{|b_{i}|}$} 
\psfrag{f}[tl]{$\scriptstyle{\psi_i}$}
\psfrag{x}[tl]{$\scriptstyle{\beta_{i-1}}$} 
\psfrag{y}[tl]{$\scriptstyle{\beta_{i}}$}
\includegraphics[width=0.6\textwidth]{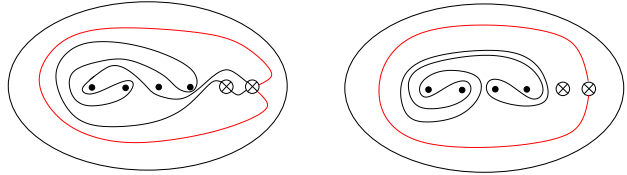}
%\resizebox{0.5\textwidth, angle=-90}{!}{\includegraphics{famcurves}}
\caption{$L^{\ast}$ is a simple closed curve on the right but it is not on the left.}\label{ucomponent1}
\end{center}
\end{figure}

\item[\textnormal{P6}.]   Some integral laminations contain $R$-components: an $R$-component of $L$ has geometric intersection numbers $i(R,\alpha_j)=1$ for each $1\leq j \leq 2n-2$,  $i(R,\beta_j)=2$ for each $1\leq j \leq n+k-1$ and $i(R,\gamma_j)=2$  for each $1\leq j \leq k-1$, which has its end points on the $k$-th crosscap (denoted red in Figure \ref{ucomponent1}). Set $L^{\ast}=L\setminus R$.  Note that $L^{\ast}$ is a component of $L$ which isn't necessarily a simple closed curve (the two possible cases are depicted in Figure \ref{ucomponent1}). Let $\alpha_i^{\ast}, \beta_i^{\ast}$ and $ \gamma_i^{\ast}$ denote the number of intersections of $L^{\ast}$ with the arcs $\alpha_i, \beta_i $ and $\gamma_i$ respectively. Define $a_i^{\ast}, b_i^{\ast},t_i^{\ast}$  and $\lambda^{\ast}_{i}$, $\lambda^{\ast}_{c_i}$, $a^{\ast}_{S'}$, $b^{\ast}_{S'}$ and $\psi^{\ast}_i$ similarly as above. We therefore have

\begin{align*}
\beta^{\ast}_1&=2\left[\left|a^{\ast}_i \right|+\max(b^{\ast}_i, 0)+m^{\ast}_i+\displaystyle\sum^{i-1}_{j=1}b^{\ast}_j \right]\quad  \text{for}\quad 1\leq i \leq n-1,\\
\beta^{\ast}_{1}&=|t^{\ast}_i|+2\max(b^{\ast}_{n+i-1}, 0)+\psi^{\ast}_{i}+2n^{\ast}_i +2\sum_{j=1}^{n+i-2}b^{\ast}_j\quad \text{for}\quad 1\leq i \leq k-1.
\end{align*} 

\noindent where $m^{\ast}_i=\min\left\{\alpha^{\ast}_{2i}-|b^{\ast}_i|, \alpha^{\ast}_{2i-1}-|b^{\ast}_i|\right\}$; $1\leq i \leq n-1$  and  $n^{\ast}_i= \min\left\{ a^{\ast}_{S'_i}, b^{\ast}_{S'_i} \right\}$;  $1\leq i \leq k-1$. Furthermore,  there is some $m^{\ast}_i=0$, or some $n^{\ast}_i=0$ since otherwise $L^{\ast}$ would have above and below components in each $S_i$ and $S'_i$ which would yield curves parallel to $\partial{N_{k,n}}$,  or $L^{\ast}$ would contain $R$-components which is impossible by definition.  Write $a_i^{\ast}=a_i, b_i^{\ast}=b_i,t_i^{\ast}=t_i$ since deleting  $R$-components doesn't change the $a,b,t$ values. Set

\begin{align*}
X_i&=2\left[\left|a_i \right|+\max(b_i, 0)+\displaystyle\sum^{i-1}_{j=1}b_j \right]\quad  \text{for}\quad 1\leq i \leq n-1,\\
Y_i&=|t_i|+2\max(b_{n+i-1}, 0)+\psi_{i}+2\sum_{j=1}^{n+i-2}b_j\quad \text{for}\quad 1\leq i \leq k-1.
\end{align*}  

\vspace{1mm}
Then one of the three following cases hold for $L^{\ast}$:
\vspace{2mm}

 \begin{enumerate}[\textnormal{I}.]
\item If $m^{\ast}_i> 0$ for all  $ 1\leq i \leq n-1$, then there is some $j$ with $1\leq j \leq k-1$ such that $n^{\ast}_j=0$.  Therefore, $\beta^{\ast}_1>X_i$ and $\beta^{\ast}_1=Y_j$. 
 \end{enumerate}
 \begin{enumerate}[\textnormal{II}.]
\item  If $n^{\ast}_i> 0$  for all  $1\leq i \leq k-1$, then there is some $j$ with $1\leq j \leq n-1$ such that $m^{\ast}_j=0$. Therefore, $\beta^{\ast}_1>Y_i$ and  $\beta^{\ast}_1=X_j$. 
 \end{enumerate}
 \begin{enumerate}[\textnormal{III}.]
\item There is some $i$ with $1\leq i \leq n-1$ such that $m^{\ast}_i=0$ and  some $j$ with $1\leq j \leq k-1$ such that 
$n^{\ast}_j=0$. Therefore, $\beta^{\ast}_1=X_i=Y_j$. 
\end{enumerate}
 We therefore have 

\begin{align*}
\beta^{\ast}_i=\max(X,Y)-2\sum_{j=1}^{i-1}b_j
\end{align*}
\noindent where $$
\displaystyle 
X=2\max_{1\leq r\leq n-1}\left\{|a_r|+\max(b_r,0)+\sum_{j=1}^{r-1}b_j\right\} 
$$
and  
$$
\displaystyle
Y=\max_{1\leq s\leq k-1}\left\{|t_s|+2\max(b_{n+s-1}, 0)+\psi_{s}+2\sum_{j=1}^{n+s-2}b_j\right\}. 
$$

\item[\textnormal{P7}.]   If $L$ doesn't have an $R$-component, that is if $L^{\ast}=L$ then $2c_k\leq \beta^{\ast}_{n+k-1}=\beta_{n+k-1}$ since $\beta_{n+k-1}=2c_k+2\lambda_k$.    If $L$ has an $R$-component then $2c_k>\beta^{\ast}_{n+k-1}$ and $\lambda_k=0$.  See  Figure \ref{ucomponent1}. Hence the number of $R$-components of $L$ is given by \begin{align*}R=\max(0,2c_k-\beta^{\ast}_{n+k-1})/2.\end{align*} 
For example, the integral laminations in Figure \ref{ucomponent1} (from left to right) has $c_1=2, \beta^{\ast}_5=2$, and hence $R=1$; and $c_1=1, \beta^{\ast}_5=0$, and hence $R=1$.  Then $L$ is constructed by identifying the two end points of an $R$ component with the pieces of $L^{\ast}$ on the $k$-th crosscap.  Since $R$-components intersect each $\beta_i$  twice we get 

\begin{align*}
\beta_i=\beta^{\ast}_i+2R; 1\leq i \leq n+k-1.
\end{align*}
Then

\begin{align*}
\beta_i &=\max(X,Y)-2\sum_{j=1}^{i-1}b_j+2R                     
\end{align*}

Also, from item P$3.$ we have

\begin{align*}
 \alpha_i     &= \left\{     
\begin{array}{lr}
(-1)^i a_{\lceil i/2 \rceil}+\frac{\beta_{\lceil i/2\rceil}}{2}& \mbox{if $b_{\lceil i/2\rceil} \geq 0$,} \\
(-1)^i a_{\lceil i/2\rceil}+\frac{\beta_{1+\lceil i/2\rceil}}{2}& \mbox{if $b_{\lceil i/2\rceil} \leq 0$,}   
\end{array}
 \right.
 \end{align*}
\noindent Finally, it is easy to observe from Figure \ref{pathcomponentsproof2} that 

\begin{align*} \gamma_i=2(a_{S'_i}+|b_{n+i-1}|+\psi_{i})\end{align*} 
\end{enumerate}

 \end{properties}

Making use of the properties above, we shall define the generalized Dynnikov coordinate system which coordinatizes $\mathcal{L}_{k,n}$ bijectively and with the least number of coordinates. 

\begin{definition}\label{dynnikovcoordinates}
The generalized Dynnikov coordinate function 
$$
\rho \colon \cL_{k,n}\to(\mathbb{Z}^{2(n+k-2)}\times \mathbb{Z}^k)\setminus \left\{0\right\}
$$ 
is defined by
$$
\displaystyle
\rho(\cL) = (a; b; t; c):=(a_1,\ldots,a_{n-1}; b_1, \ldots,b_{n+k-2}; t_1, \ldots,t_{k-1}; c_1, \ldots, c_k)$$
\noindent where 
\begin{align*}
\displaystyle
a_i&=\frac{\alpha_{2i}-\alpha_{2i-1}}{2} &&\textrm{for} \ 1\leq i\leq n-1,\\
\displaystyle
b_i&=\frac{\beta_i-\beta_{i+1}}{2} &&\textrm{for} \ 1\leq i\leq n+k-2,\\
\displaystyle
t_i&=a_{S'_i}-b_{S'_ i} &&\textrm{for} \ 1\leq i\leq k-1,
\end{align*}
where $a_{S'_i}$ and $b_{S'_i}$ are as given in Lemma \ref{abovebelow3}.

\end{definition}

Theorem \ref{thm:dynninvert} gives the inverse of $\rho \colon \cL_{k,n}\to(\mathbb{Z}^{2(n+k-2)}\times \mathbb{Z}^k)\setminus \left\{0\right\}$.

\begin{theorem}\label{thm:dynninvert}
Let $(a; b; t; c)\in(\mathbb{Z}^{2(n+k-2)}\times \mathbb{Z}^k)\setminus \left\{0\right\}$. Set 

\vspace{-0.5cm}

\begin{align*}
X&=2\max_{1\leq r\leq n-1}\left\{|a_r|+\max(b_r,0)+\sum_{j=1}^{r-1}b_j\right\}\\
Y&=\max_{1\leq s\leq k-1}\left\{|t_s|+2\max(b_{n+s-1}, 0)+\psi_{s}+2\sum_{j=1}^{n+s-2}b_j\right\}. \end{align*} 
Then $(a; b; t; c)$ is the Dynnikov coordinate of exactly one element $\cL \in \cL_{k, n}$ which has 
\vspace{-0.2cm}
\begin{align}
\displaystyle \beta_i      &=\max(X,Y)-2\sum_{j=1}^{i-1}b_j+2R,   \\                                                   
\displaystyle \alpha_i     &= \left\{     
\begin{array}{lr}
(-1)^i a_{\lceil i/2 \rceil}+\frac{\beta_{\lceil i/2\rceil}}{2}& \mbox{if $b_{\lceil i/2\rceil} \geq 0$,} \\
(-1)^i a_{\lceil i/2\rceil}+\frac{\beta_{1+\lceil i/2\rceil}}{2}& \mbox{if $b_{\lceil i/2\rceil} \leq 0$,}   
\end{array}
 \right.\\
 \gamma_i&=2(a_{S'_i}+|b_{n+i-1}|+\psi_{i}) 
 \end{align}
\noindent where $a_{S'_i}$ is defined as in item P$4.$ in Properties \ref{properties}.

 \end{theorem}

\begin{proof}

Given $L\in \mathcal{L}_{k,n}$ with $\tau(L)=(\alpha,\beta,\gamma, c)$ and $\rho(L)=(a,b,t,c)$, Properties \ref{properties} show that $\alpha,\beta$ and $\gamma$ must be given by (2), (3) and (4) respectively, and hence  $L$ is unique by Lemma \ref{tauinject}. Therefore $\rho$ is injective.  By Properties \ref{properties} we can draw non-intersecting path components in each $S_i$ ($1\leq i\leq n-1$), $S'_i$ ($1\leq i\leq k-1$), $\Delta_0$ and $\Delta'_k$ which intersect each element of $\mathcal{A}_{k,n}$ the number of times given by $(\alpha,\beta,\gamma, c)$. Gluing together these path components gives a disjoint union of simple closed curves in $N_{k,n}$. There are no curves that bound a puncture or parallel to the boundary by construction, and hence $(\alpha,\beta,\gamma, c)$ where $\alpha,\beta$ and $\gamma$ are defined by (2), (3) and (4) respectively, correspond to some $L$ with $\rho(L)=(a,b,t,c)$. Therefore, $\rho$ is surjective.\end{proof}

\begin{example}\label{secondexample}
Let $\rho(\cL)=(a_1; b_1, b_2; t_1; c_1, c_2)=(-1; 2, 0; 1; 1, 0)$ be the generalized Dynnikov coordinates of an integral lamination $\cL$ on $N_{2, 2}$.  We shall use Theorem \ref{thm:dynninvert} to compute the triangle coordinates of $\cL$ from which we determine the number of path components in $S_1$ and $S'_1$, and hence draw $\cL$ as illustrated in Example \ref{firstexample}. By Lemma \ref {abovebelow2}, $\psi_{1}=\max(c_1^+-|b_{2}|, 0)=1$ so we have                         
\begin{align*}
X=2(|a_1|+\max(b_1, 0))=6  \quad \text{and} \quad  Y= |t_1|+2\max(b_2, 0)+\psi_{1}+2b_1=6.
\end{align*}
Therefore
\begin{align*}
\beta_1&=\max(6, 6)=6, \ \beta_2=\max(6, 6)-2b_1=2, \ \beta_3=\max(6, 6)-2(b_1+b_2)=2, \\
\alpha_1&= -a_1+\frac{\beta_1}{2}=4, \ \alpha_2=a_1+\frac{\beta_1}{2}=2.
 \end{align*}
Since  $0=2c_2 < \beta^{\ast}_3=2$, there are no $R$-components by item P$8.$ of Properties \ref{properties}. Since $\beta_1=6$ there are 3 loop components in $\Delta_0$, and since $\beta_3=2$ and $c_2=0$, there is one non-core loop component in $\Delta'_2$ that is $\lambda_2=1$. By Remarks \ref{sumofloop}, $b_1=2$ and $b_2=0$, and hence there are 2 right loop components in $S_1$ and no loop components in $S'_1$. By Lemma \ref{abovebelow} we compute that $a_{S_1}=\alpha_{1}-|b_1|=2$ and $b_{S_1}=\alpha_{2}-|b_1|=0$. Finally by item P$4.$ of Properties \ref{properties},
\begin{align*}
 a_{S'_1}&=\frac{t_1-\psi_{1}+\max(\beta_2, \beta_3)-2|b_2|}{2}=1\\
  b_{S'_1} &=\frac{-t_1-\psi_{1}+\max(\beta_{2}, \beta_{3})-2|b_{2}|}{2}=0
\end{align*}
Gluing together the path components in $S_1$ and $S'_1$ we construct the integral lamination depicted in Figure \ref{construct}.

\begin{figure}[h!]
\begin{center}
\psfrag{1}[tl]{$\scriptstyle{1}$} 
\psfrag{2}[tl]{$\scriptstyle{2}$} 
\psfrag{4}[tl]{$\scriptstyle{4}$} 
\psfrag{10}[tl]{$\scriptstyle{10}$}
\psfrag{6}[tl]{$\scriptstyle{6}$} 
\psfrag{8}[tl]{$\scriptstyle{8}$}
\includegraphics[width=0.32\textwidth]{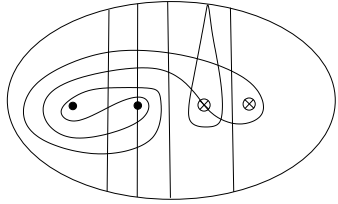}
\caption{$\rho(L)=(-1; 2, 0; 1; 1, 0)$}\label{construct}
\end{center}
\end{figure}
\end{example}

\begin{remark}
 
Generalized Dynnikov coordinates for integral laminations can be extended in a natural way to generalized Dynnikov coordinates of measured foliations \cite{HY08}: the transverse measure on the foliation \cite{W88, FLP79, pepa} assigns to each element in $\mathcal{A}_{k,n}$ a non-negative real number, and hence each measured foliation is described by an element of $(\mathbb{R}_{\geq 0}^{3n+2k-4}\times \mathbb{R}^k)\setminus\left\{0\right\}$, the associated measures of the arcs and curves of $\mathcal{A}_{k,n}$. Therefore, the \emph{Generalized Dynnikov coordinate system} for measured foliations is defined similarly (see Definition \ref{dynnikovcoordinates}), and provides a one-to-one correspondence between the set of measured foliations (up to isotopy and Whitehead equivalence) on $N_{k,n}$ and $(\mathbb{R}^{2(n+k-2)}\times \mathbb{R}^k)\setminus \left\{0\right\}$.

\end{remark}

%%%%%%%%%%%%%%%%%%%%%%%%%%%%%%%%%%%%%%%%%%%%%%%%%%%%%%%%%%%%%%%%%%%%%%%%%%%%%%%%%%%%%%%%%%%%%%%%%%%%%%%%%%%%%%%%%%%%%%%%%%%%%%%%%%


\begin{thebibliography}{10}

\bibitem{or08}
Dehornoy, Patrick and Dynnikov, Ivan and Rolfsen, Dale and Wiest, Bert,
\emph{Ordering braids},
volume $148$ of Mathematical Surveys and Monographs. American Mathematical
Society, Providence, RI, 2008.


\bibitem{D02}
I.~Dynnikov,
\emph{On a Yang-Baxter mapping and the Dehornoy ordering},
Uspekhi Mat.  Nauk,  57(3(345)): 151--152, 2002.

\bibitem{wiest}
Dynnikov, I. and Wiest, B.,
\emph{On the complexity of braids},
J.  Eur.  Math.  Soc. (JEMS), 9(4):801--840, 2007


\bibitem{FLP79}
Fathi,A. and Laudenbach,F. and Poenaru,V.
\emph{Travaux de {T}hurston sur les surfaces},
Soci\'et\'e Math\'ematique de France,(66) 284, 1979



\bibitem{HY08}
Hall, Toby and Yurttas,S.~Oyku,
\emph{On the topological entropy of families of braids},
Topology Appl., 156(8):1554--1564, 2009.


\bibitem{M06}
Moussafir, J-O.,
\emph{On computing the entropy of braids},
Funct. Anal. Other Math., 1(1):37--46, 2006.

\bibitem{pepa}
Papadopoulos, A. and Penner, R.~C.,
\emph{Hyperbolic metrics, measured foliations and pants decompositions for non-orientable surfaces},
Asian J. Math., 20(1):157--182, 2016.

\bibitem{W88}
Thurston, W.~P.,
\emph{On the geometry and dynamics of diffeomorphisms of surfaces},
Bull. Amer. Math. Soc. (N.S.), 19(2):417--431, 1988.

\bibitem{paper2}
Yurtta{\c{s}}, S.~{\"O}yk{\"u},
\emph{Geometric intersection of curves on punctured disks},
Journal of the Mathematical Society of Japan, 65(4):1554--1564, 2013.

\bibitem{paper3}
Yurtta{\c{s}}, S.~{\"O}yk{\"u},
\emph{Dynnikov and train track transition matrices of pseudo-{A}nosov braids},
Discrete Contin. Dyn. Syst., 36(1):541--570, 2016.

\bibitem{HY16}
Yurtta{\c{s}}, S.~{\"O}yk{\"u} and Hall, Toby,
\emph{Counting components of an integral lamination},
manuscripta math. doi:10.1007/s00229-016-0885-4, 2016.

\end{thebibliography}
\end{document}